\documentclass{amsart}
\usepackage{amstext,amssymb,amsthm,amsopn,newlfont,graphpap,graphics,graphicx,mathrsfs,enumitem}
\allowdisplaybreaks
\usepackage[parfill]{parskip}
\usepackage[noadjust]{cite}
\usepackage{epigraph}
\usepackage[colorlinks=true,
            linkcolor=red,
            urlcolor=blue,
            citecolor=magenta]{hyperref}
\usepackage{color}
\usepackage{mathrsfs}
\allowdisplaybreaks
\theoremstyle{plain}
\newtheorem{thm}{Theorem}[section]
\newtheorem*{thm*}{Theorem}
\newtheorem{prop}{Proposition}[section]
\newtheorem*{prop*}{Proposition}
\newtheorem{cor}{Corollary}[section]
\newtheorem*{cor*}{Corollary}
\newtheorem{lem}{Lemma}[section]
\newtheorem*{lem*}{Lemma}
\theoremstyle{definition}
\newtheorem{defn}{Definition}[section]
\newtheorem*{defn*}{Definition}

\newtheorem*{exmps*}{Examples}

\newtheorem*{exmp*}{Example}

\newtheorem*{exerc*}{Exercise}

\newtheorem{rems}{Remarks}[section]
\newtheorem*{rems*}{Remarks}
\newtheorem{rem}[rems]{Remark}
\newtheorem*{rem*}{Remark}
\newcommand{\N}{{\mathbb N}}
\newcommand{\Z}{{\mathbb Z}}
\newcommand{\R}{{\mathbb R}}
\newcommand{\C}{{\mathbb C}}

\renewcommand{\iff}{\: \Leftrightarrow\: }
\renewcommand{\bar}{\overline}
\DeclareMathOperator{\Rep}{Re\ignorespaces}
\DeclareMathOperator{\Imp}{Im\ignorespaces}
\DeclareMathOperator{\dist}{dist}

\DeclareMathOperator*{\esup}{\mbox{$E_A$}-ess\,sup}

\begin{document}
\title[On the regularity of scalar type spectral $C_0$-semigroups]
{On the regularity\\ of scalar type spectral $C_0$-semigroups}
\author[Marat V. Markin]{Marat V. Markin}
\address{
Department of Mathematics\newline
California State University, Fresno\newline
5245 N. Backer Avenue, M/S PB 108\newline
Fresno, CA 93740-8001, USA
}
\email{mmarkin@csufresno.edu}
\begin{abstract}
We show that, for the $C_0$-semigroups of \textit{scalar type spectral operators}, a well-known necessary condition for the generation of \textit{eventually norm-continuous} 
$C_0$-semigroups, formulated exclusively in terms of the location of the spectrum of the semigroup's generator in the complex plane, is also sufficient and, in fact, characterizes the generators of \textit{immediately norm-continuous}  such semigroups. 

Combining characterizations of the \textit{immediate differentiability} and the \textit{Gev\-rey ultradifferentiability} of scalar type spectral $C_0$-semigroups with the \textit{generation theorem}, found earlier by the author, we arrive at respective characterizations of the generation of such semigroups.

We further establish characterizations of the generation of \textit{eventually differentiable} and \textit{immediately compact} scalar type spectral $C_0$-semigroups also in terms of the generator's spectrum and show that, for such semigroups, eventual compactness implies immediate. 

All the obtained results are instantly transferred to the $C_0$-semigroups of \textit{normal operators}.
\end{abstract}
\maketitle

\section[Introduction]{Introduction}

While a Hille–Yosida type characterization of the generation of \textit{eventually norm continuous} $C_0$-semigroups (see Preliminaries) appears to be still unestablished (see, e.g.,  \cite{Blasco-Martinez1996,El-Mennaoui-Engel1996,Goersmeyer-Weis1999,Engel-Nagel,Engel-Nagel2006}, for $C_0$-semigroups on Hilbert spaces, see \cite{You1992,El-Mennaoui-Engel1994} and \cite{Engel-Nagel}), well-known is the following necessary condition formulated exclusively in terms of the location of the spectrum of the semigroup's generator in the complex plane.

\begin{thm}[Necessary Condition for the Generation of Eventually Norm-Continuous $C_0$-Semigroups {\cite[Theorem II.$5.3$]{Engel-Nagel2006}}]\label{NCENC}
If $A$ is a generator of an eventually norm continuous $C_0$-semigroup on a complex Banach space, then, for any $b\in \R$, the set
\begin{equation*}
\left\{\lambda\in\sigma(A) \,\middle|\, \Rep\lambda\ge b\right\}
\end{equation*}
($\sigma(\cdot)$ is the \textit{spectrum} of an operator) is bounded.
\end{thm}

To a significant extent, the importance of eventually norm-continuous $C_0$-semigroups
is attributed the fact, that such a semigroup $\left\{T(t) \right\}_{t\ge 0}$ with generator $A$ is subject to a \textit{weak spectral mapping theorem} 
\begin{equation*}\tag{WSMT}\label{WSMT}
\sigma(T(t))\setminus \{0\}=\overline{e^{t\sigma(A)}}\setminus \{0\},\ t\ge 0,
\end{equation*}
($\overline{\cdot}$ is the \textit{closure} of a set)  
\cite[Examples V.$2.2$]{Engel-Nagel2006}, in fact, 
a stronger \textit{spectral mapping theorem}
\begin{equation*}\tag{SMT}\label{SMT}
\sigma(T(t))\setminus \{0\}=e^{t\sigma(A)},\ t\ge 0,
\end{equation*}
\cite[Theorem V.$2.8$, Proposition  V.$2.3$]{Engel-Nagel2006}, and with that the following \textit{spectral bound equal growth bound condition} 
\begin{equation*}\tag{SBeGB}\label{SBeGB}
s(A)=\omega_0,
\end{equation*}
where
\[
s(A):=\sup\left\{\Rep\lambda \,\middle|\,\lambda\in\sigma(A) \right\}
\ (s(A):=-\infty\ \text{if}\ \sigma(A)=\emptyset)
\] 
is the \textit{spectral bound} of the generator and
\begin{equation}\label{gb}
\omega_0:=\inf\left\{\omega\in\R \,\middle|\,\exists\,M=M(\omega)\ge 1:\ \|T(t)\|\le M e^{\omega t},\ t\ge 0 \right\}
\end{equation}
is the \textit{growth bound} of the semigroup \cite[Proposition  V.$2.3$, Corollary  V.$2.10$]{Engel-Nagel2006}. The asymptotic behavior of a $C_0$-semigroup satisfying the \textit{spectral bound equal growth bound condition} \eqref{SBeGB} is governed by the spectral bound of its generator and, in particular, is subject to a \textit{generalized Lyapunov stability theorem} \cite[Theorem V.$3.6$]{Engel-Nagel2006}.

Eventually norm-continuous are $C_0$-semigroups with certain regularity properties such as \textit{eventual differentiability}, in particular \textit{analyticity} or \textit{uniform continuity}, or \textit{eventual compactness} \cite{Pazy1968,Yosida1958} (see also \cite{Engel-Nagel2006}).

As is shown in \cite{Markin2021(2)}, a \textit{scalar type spectral} $C_0$-semigroup $\left\{T(t) \right\}_{t\ge 0}$ with generator $A$ is subject to the following \textit{weak spectral mapping theorem}
\begin{equation*}
\sigma(T(t))=\overline{e^{t\sigma(A)}},\ t\ge 0,
\end{equation*}
\cite[Theorem $4.1$]{Markin2021(2)} (cf. \cite[Corollary V.$2.12$]{Engel-Nagel2006})
and with that \textit{spectral bound equal growth bound condition} \eqref{SBeGB} without any regularity requirement.

We show that, for the $C_0$-semigroups of \textit{scalar type spectral operators}, the above necessary condition is also sufficient  and, in fact, characterizes the generators of \textit{immediately norm continuous} such semigroups. 

Combining the characterizations of the \textit{immediate differentiability} and \textit{Gevrey ultradifferentiability} of scalar type spectral $C_0$-semigroups 
found in \cite{Markin2002(2),Markin2004(1),Markin2016} (see also \cite{Markin2008}) with the \textit{generation theorem} \cite[Proposition $3.1$]{Markin2002(2)}, we arrive at respective characterizations of the generation of such semigroups.

We further establish characterizations of the generation of \textit{eventually differentiable} and \textit{immediately compact} scalar type spectral $C_0$-semigroups also in terms of the generator's spectrum and show that, for such semigroups, eventual compactness implies immediate. 

All the obtained results are immediately transferred to the $C_0$-semigroups of \textit{normal operators}.

This work continues the series of papers \cite{Markin2002(2),Markin2004(1),Markin2008,Markin2016} (see also \cite{Markin2004(2),Markin2015}) on the generation of  differentiabile and Carleman, in particular Gevrey, ultradifferentiabile scalar type spectral $C_0$-semigroups .
 
\section[Preliminaries]{Preliminaries}

Here, we concisely outline preliminaries necessary for our subsequent discourse.

\subsection{$C_0$-Semigroups}\

\begin{defn}[Norm-Continuity]\ \\
A $C_0$-semigroup $\left\{T(t)\right\}_{t\ge 0}$ on a Banach space $(X,\|\cdot\|)$ is said to be
\begin{itemize}
\item \textit{eventually norm-continuous} if there exists a $t_0>0$ such that the operator function
\[
[t_0,\infty)\ni t\mapsto T(t)\in L(X)
\]
($L(X)$ is the space of bounded linear operators on $X$)
is continuous relative to the operator norm;
\item \textit{immediately norm-continuous} if the operator function
\begin{equation*} 
(0,\infty)\ni t\mapsto T(t)\in L(X)
\end{equation*} 
is continuous relative to the operator norm;
\item \textit{uniformly continuous} if the operator function
\begin{equation*} 
[0,\infty)\ni t\mapsto T(t)\in L(X)
\end{equation*} 
is continuous relative to the operator norm.
\end{itemize}
\end{defn}

\begin{defn}[Differentiability]\ \\
A $C_0$-semigroup $\left\{T(t)\right\}_{t\ge 0}$ on a Banach space $(X,\|\cdot\|)$ is said to be
\begin{itemize}
\item \textit{eventually differentiable} if there exists a $t_0>0$ such that the semigroup is strongly differentiable on $[t_0,\infty)$, i.e., the orbit maps
\begin{equation*}
[0,\infty)\ni t\mapsto T(t)f\in X
\end{equation*} 
are differentiable on $[t_0,\infty)$ for all $f\in X$;
\item \textit{immediately differentiable} if the semigroup is strongly differentiable on $(0,\infty)$, i.e., the orbit maps
\begin{equation*}
[0,\infty)\ni t\mapsto T(t)f\in X
\end{equation*} 
are differentiable on $(0,\infty)$ for all $f\in X$.
\end{itemize}
\end{defn}

\begin{samepage}
\begin{rems}\label{remsed}\
\begin{itemize}
\item If a $C_0$-semigroup is strongly differentiable for $t\ge t_0$, it is \textit{norm-continuous} for $t\ge t_0$ \cite[Lemma $2.1$]{Pazy1968} (see also \cite[Theorem $10.3.5$]{Hille-Phillips}).
\item An immediately differentiable $C_0$-semigroup  $\left\{T(t)\right\}_{t\ge 0}$, often referred to as a $C^\infty$-semigroup, is \textit{immediately norm-continuous}, and furthermore, \textit{infinite differentiable} on $(0,\infty)$ relative to the operator norm with
\[
T^{(n)}(t)=A^nT(t),\ t>0,
\]
\cite[Lemma $2.1$]{Pazy1968}.
\end{itemize}
\end{rems}
\end{samepage}

\begin{defn}[Compactness]\ \\
A $C_0$-semigroup $\left\{T(t)\right\}_{t\ge 0}$ on a Banach space $(X,\|\cdot\|)$ is said to be
\begin{itemize}
\item \textit{eventually compact} if there exists a $t_0>0$ such that the operator $T(t_0)$ is compact, and hence, $T(t)$ is compact for all $t\ge t_0$ (see, e.g., \cite{Markin2020EOT});
\item \textit{immediately compact} if the operator $T(t)$ is compact for all $t>0$.
\end{itemize}
\end{defn}

\begin{rems}\label{remscs}\
\begin{itemize}
\item If a $C_0$-semigroup is compact for $t\ge t_0$, it is \textit{norm-continuous} for $t\ge t_0$ \cite[Theorem $3.1$]{Pazy1968}.

Thus, an immediately compact $C_0$-semigroup is \textit{immediately norm-conti\-nuous}.
\item Observe that a \textit{bounded} operator $A$ in a complex infinite-dimensional Banach space cannot generate a \textit{compact} $C_0$-semigroup $\left\{T(t)\right\}_{t\ge 0}$, which would be \textit{uniformly continuous}, and hence, the identity operator $I$ would be \textit{compact} as the uniform limit of $T(t)$ as $t\to 0+$. The latter would contradict the infinite-dimensionality of $X$ 
(see, e.g., \cite{Markin2020EOT}).
\end{itemize}
\end{rems}

For a linear operator $A$ generating a $C_0$-semigroup $\left\{ T(t)\right\}_{t\ge 0}$ in a complex Banach space $(X,\|\cdot\|)$ the \textit{abstract Cauchy problem}
\begin{equation}\tag{ACP}\label{ACP}
\begin{cases}
y'(t)=Ay(t),\ t\ge 0,\\
y(0)=f\in X,
\end{cases}
\end{equation}
is \textit{well-posed}, in particular \textit{uniquely solvable}, in which case the \textit{weak solutions} (also called \textit{mild solutions}) of the associated \textit{abstract evolution equation} 
\begin{equation*}\tag{AEE}\label{AEE}
y'(t)=Ay(t),\ t\ge 0,
\end{equation*}
are the orbit maps
\begin{equation*}
y(t)=T(t)f,\ t\ge 0,
\end{equation*}
with $f\in X$ \cite{Ball,Engel-Nagel,Engel-Nagel2006}, whereas the \textit{classical} ones are the orbit maps with $f\in D(A)$.

\subsection{Scalar Type Spectral Operators}\

A {\it scalar type spectral operator} is a densely defined closed linear operator $A$ in a complex Banach space $(X,\|\cdot\|)$ with strongly $\sigma$-additive \textit{spectral measure} (the \textit{resolution of the identity}) $E_A(\cdot)$, assigning to Borel sets of the complex plane $\C$ projection operators on $X$ and having the operator's \textit{spectrum} $\sigma(A)$ as its {\it support}, and the associated {\it Borel operational calculus}, assigning to each Borel measurable function $F:\sigma(A)\to \bar{\C}$ ($\bar{\C}:=\C\cup \{\infty\}$ is the extended complex plane) with $E_A\left(\left\{\lambda\in \C \,\middle|\, F(\lambda)=\infty\right\}\right)=0$ (here and whenever appropriate, $0$ designates the \textit{zero operator} on $X$) a scalar type spectral operator in $X$
\begin{equation*}
F(A):=\int\limits_{\sigma(A)} F(\lambda)\,dE_A(\lambda)
\end{equation*} 
with
\begin{equation*}
A=\int\limits_{\sigma(A)} \lambda\,dE_A(\lambda)
\end{equation*}
\cite{Dunford1954,Survey58,Dun-SchIII}.

In a complex finite-dimensional space, 
the scalar type spectral operators are those linear operators on the space, which furnish an \textit{eigenbasis} for it (see, e.g., \cite{Survey58,Dun-SchIII}) and, in a complex Hilbert space, the scalar type spectral operators are those that are similar to the {\it normal operators} \cite{Wermer1954}.

Due to its strong $\sigma$-additivity, the spectral measure is uniformly bounded, i.e.,
\begin{equation}\label{bounded}
\exists\, M\ge 1\ \forall\,\delta\in \mathscr{B}(\C):\ \|E_A(\delta)\|\le M
\end{equation}
($\mathscr{B}(\C)$ is the \textit{Borel $\sigma$-algebra} on $\C$)
(see, e.g., \cite{Dun-SchI}).

\begin{rem}
Here and henceforth, we use the same notation $\|\cdot\|$ for the operator norm and adhere to this rather common economy of symbols also for the functional norm in the \textit{dual space} $X^*$.
\end{rem}

By \cite[Theorem XVIII.$2.11$ (c)]{Dun-SchIII}, for a Borel measurable function $F:\sigma(A)\to \bar{\C}$, the operator $F(A)$ is \textit{bounded} iff
\[
\esup_{\lambda\in \sigma(A)}|F(\lambda)|<\infty,
\]
in which case, we have the estimates
\begin{equation}\label{boundedop}
\esup_{\lambda\in \sigma(A)}|F(\lambda)|\le \|F(A)\|
\le 4M\esup_{\lambda\in \sigma(A)}|F(\lambda)|,
\end{equation}
where $M\ge 1$ is from \eqref{bounded}.

In particular, this implies that a scalar type spectral operator is \textit{bounded} iff its spectrum $\sigma(A)$ is a \textit{bounded} set.

For arbitrary Borel measurable function $F:\C\to \C$, $f\in D(F(A))$, $g^*\in X^*$, and Borel set $\delta\subseteq \C$,
\begin{equation}\label{cond(ii)}
\int\limits_\delta|F(\lambda)|\,dv(f,g^*,\lambda)
\le 4M\|E_A(\delta)F(A)f\|\|g^*\|,
\end{equation}
where $v(f,g^*,\cdot)$ is the \textit{total variation measure} of the complex-valued Borel measure $\langle E_A(\cdot)f,g^* \rangle$, for which
\begin{equation}\label{tv}
v(f,g^*,\C)=v(f,g^*,\sigma(A))\le 4M\|f\|\|g^*\|,
\end{equation}
where $M\ge 1$ in \eqref{cond(ii)} and \eqref{tv}
is from \eqref{bounded} (see, e.g., \cite{Markin2004(1),Markin2004(2)}). 

The following statement allows to characterize the domains of the Borel measurable functions of a scalar type spectral operator in terms of positive Borel measures.

\begin{prop}[{\cite[Proposition $3.1$]{Markin2002(1)}}]\label{prop}\ \\
Let $A$ be a scalar type spectral operator in a complex Banach space $X$ with spectral measure $E_A(\cdot)$ and $F:\sigma(A)\to \C$ be a Borel measurable function. Then $f\in D(F(A))$ iff
\begin{enumerate}[label={(\roman*)}]
\item for each $g^*\in X^*$, 
$\displaystyle \int\limits_{\sigma(A)} |F(\lambda)|\,d v(f,g^*,\lambda)<\infty$ and
\item $\displaystyle \sup_{\{g^*\in X^*\,|\,\|g^*\|=1\}}
\int\limits_{\{\lambda\in\sigma(A)\,|\,|F(\lambda)|>n\}}
|F(\lambda)|\,dv(f,g^*,\lambda)\to 0,\ n\to\infty$,
\end{enumerate}
where $X^*$ is the dual space and  $v(f,g^*,\cdot)$ is the total variation measure of $\langle E_A(\cdot)f,g^* \rangle$.
\end{prop} 

In \cite{Markin2002(2)}, it is shown that, for scalar type spectral $C_0$-semigroups, the classical \textit{Hille-Yosida-Feller-Miyadera-Phillips
generation theorem} (see, e.g., \cite[Theorems II.$3.5$ and II.$3.8$]{Engel-Nagel2006}) acquires the following inherently qualitative form exclusively in terms of the location of te generator's spectrum in the complex plane, void of any estimates on the generator's resolvent.

\begin{thm}[Generation Theorem {\cite[Proposition $3.1$]{Markin2002(2)}}]\label{GT}\ \\
A scalar type spectral operator $A$ in a complex Banach space with spectral measure $E_A(\cdot)$ generates a $C_0$-semigroup (of scalar type spectral operators) $\left\{T(t) \right\}_{t\ge 0}$ iff there exists an $\omega \in \R$ such that
\[
\sigma(A)\subseteq \left\{\lambda\in\C \,\middle|\, \Rep\lambda\le \omega\right\}
\] 
in which case
\[
T(t)=e^{tA}:=\int\limits_{\sigma(A)} e^{t\lambda}\,dE_A(\lambda),\ t\ge 0,
\]
\end{thm}

Furthermore, scalar type spectral $C_0$-semigroups are generated by scalar type spectral operators \cite{Berkson1966,Panchapagesan1969}.

In \cite{Markin2002(2),Markin2004(1),Markin2008,Markin2016}, characterizations of
the \textit{analyticity}, \textit{immediate differentiability}, and \textit{Carleman}, in particular \textit{Gevrey ultradifferentiability}, of scalar type spectral $C_0$-semigroups are found also purely in terms of the location of the generator's spectrum in the complex plane. 
In particular, the characterizations \textit{analyticity} and \textit{immediate differentiability},
similarly to the above \textit{generation theorem}, are obtained from their general Hille–Yosida type counterparts \cite{Pazy1968,Yosida1958} (see also \cite{CrPaz1969})
by dropping the corresponding estimates on the generator's resolvent.


\subsection{Finer Spectrum Structure}\

The spectrum $\sigma(A)$ of a \textit{closed linear operator} $A$ in a complex Banach space $(X,\|\cdot\|)$ is partitioned into the following pairwise disjoint subsets:
\begin{equation*}
\begin{split}
& \sigma_p(A):=\left\{\lambda\in \C \,\middle|\,A-\lambda I\ \text{is \textit{not one-to-one}, i.e., $\lambda$ is an \textit{eigenvalue} of $A$} \right\},\\
& \sigma_c(A):=\left\{\lambda\in \C \,\middle|\,A-\lambda I\ \text{is \textit{one-to-one} and $R(A-\lambda I)\neq X$, but $\overline{R(A-\lambda I)}=X$} \right\},\\
& \sigma_r(A):=\left\{\lambda\in \C \,\middle|\,A-\lambda I\ \text{is \textit{one-to-one} and $\overline{R(A-\lambda I)}\neq X$} \right\}
\end{split}
\end{equation*}
($R(\cdot)$ is the \textit{range} of an operator), called the \textit{point}, \textit{continuous}, and \textit{residual spectrum} of $A$, respectively (see, e.g., \cite{Markin2020EOT}).
 
For a scalar type spectral operator $A$,
\begin{equation}\label{rspg}
\sigma_r(A)=\emptyset
\end{equation}
\cite[Corollary 3.1]{Markin2017} (see also \cite{Markin2006}).

\section{Norm Continuity}

\begin{thm}[Characterization of the Generation of Immediately Norm-Continuous Scalar Type Spectral $C_0$-Semigroups]\label{CGINCS}
A scalar type spectral operator $A$ in a complex Banach space generates an immediately norm continuous $C_0$-semigroup (of scalar type spectral operators) iff, for any $b\in \R$, the set
\begin{equation}\label{cinc}
\left\{\lambda\in\sigma(A) \,\middle|\, \Rep\lambda\ge b\right\}
\end{equation}
is bounded.
\end{thm}

\begin{proof}\

{\it ``Only if"} part immediately follows from the \textit{Necessary Condition for the Generation of Eventually Norm-Continuous $C_0$-Semigroups} (Theorem \ref{NCENC}).

\smallskip
{\it ``If"} part. Suppose that a scalar type spectral operator $A$ in a complex Banach space with spectral measure $E_A(\cdot)$ satisfies condition \eqref{cinc}. Then there exists an $\omega \in \R$ such that
\[
\sigma(A)\subseteq \left\{\lambda\in\C \,\middle|\, \Rep\lambda\le \omega\right\},
\] 
and hence, by the \textit{Generation Theorem} (Theorem \ref{GT}),
$A$ generates a scalar type spectral $C_0$-semigroup
\[
T(t)=e^{tA}:=\int\limits_{\sigma(A)} e^{t\lambda}\,dE_A(\lambda),\ t\ge 0.
\]

The premise also implies that, for each $n\in \N$ ($\N:=\left\{ 1,2,\dots\right\}$ is the set of \textit{natural numbers}),  by the properties of the Borel operational calculus \cite[Theorem XVIII.$2.11$ (f)]{Dun-SchIII},
\begin{equation*}
A_n:=AE_A\left(\left\{\lambda\in\sigma(A) \,\middle|\, \Rep\lambda\ge -n\right\}\right)=\int\limits_{\sigma(A)}\lambda \chi_{\left\{\lambda\in\sigma(A) \,\middle|\, \Rep\lambda\ge -n\right\}}(\lambda)\,dE_A(\lambda)
\end{equation*}
($\chi_\delta(\cdot)$ is the {\it characteristic function} of a set $\delta\subseteq \C$)
is a \textit{bounded} scalar type spectral operator since its spactrum $\sigma(A_n)$ is a bounded set (see Preliminaries). Indeed, by the \textit{A.E. Weak Spectral Inclusion and Mapping Theorem} (\cite[Theorem $3.1$]{Markin2021(2)}), we have the inclusion
\begin{equation*}
\sigma(A_n)\subseteq \left\{\lambda\in\sigma(A) \,\middle|\, \Rep\lambda\ge -n\right\}\cup
\left\{0\right\}
\end{equation*}
(see also \cite{Markin2020(2)}).

Thus,  for each $n\in \N$, $A_n$ generates the \textit{uniformly continuous} semigroup $\left\{ e^{tA_n}\right\}_{t\ge 0}$ of its exponentials (see Preliminaries). 

By the properties of the Borel operational calculus \cite[Theorem XVIII.$2.11$]{Dun-SchIII},
\begin{align*}
e^{tA_n}&=
\int\limits_{\sigma(A)} e^{t\lambda \chi_{\left\{\lambda \in \sigma(A)\,
\middle|\,\Rep \lambda\ge -n\right\}}(\lambda)}\,dE_A(\lambda)
=\int\limits_{\sigma(A)} e^{t\lambda}\chi_{\left\{\lambda\in\sigma(A) \,\middle|\, \Rep\lambda\ge -n\right\}}(\lambda)\,dE_A(\lambda)
\\
&+\int\limits_{\sigma(A)} \chi_{\left\{\lambda\in\sigma(A) \,\middle|\, \Rep\lambda< -n\right\}}(\lambda)\,dE_A(\lambda)
\\
&=e^{tA}E_A\left(\left\{\lambda \in \sigma(A)\,\middle|\,\Rep\lambda\ge -n\right\}\right)
+E_A\left(\left\{\lambda \in \sigma(A)\,\middle|\,\Rep\lambda<-n\right\}\right),\  t\ge 0.
\end{align*}

Hence, for each $n\in \N$, by the properties of the Borel operational calculus
\cite[Theorem XVIII.$2.11$ (g)]{Dun-SchIII}, the operator function
\begin{align*}
[0,\infty)\ni t\mapsto &
E_A\left(\left\{\lambda\in\sigma(A) \,\middle|\, \Rep\lambda \ge  -n\right\}\right)e^{tA_n}
\\
&=E_A\left(\left\{\lambda \in \sigma(A)\,
\middle|\,\Rep\lambda\ge -n\right\}\right)e^{tA}E_A\left(\left\{\lambda \in \sigma(A)\,
\middle|\,\Rep\lambda\ge -n\right\}\right)\\
&=e^{tA}E_A\left(\left\{\lambda \in \sigma(A)\,
\middle|\,\Rep\lambda\ge -n\right\}\right)
\in L(X)
\end{align*} 
is \textit{continuous} relative to the \textit{operator norm} on $[0,\infty)$.

For arbitrary $t_0>0$ and $n\in \N$, we have:
\begin{multline*}
\sup_{t\ge t_0}\|T(t)-E_A\left(\left\{\lambda\in\sigma(A) \;\middle|\; \Rep\lambda \ge  -n\right\}\right)e^{tA_n}\|
\\
\shoveleft{
=\sup_{t\ge t_0}\|e^{tA}-E_A\left(\left\{\lambda\in\sigma(A) \;\middle|\; \Rep\lambda \ge  -n\right\}\right)e^{tA_n}\|
}\\
\shoveleft{
=\sup_{t\ge t_0}\|e^{tA}-e^{tA}E_A\left(\left\{\lambda \in \sigma(A)\,
\middle|\,\Rep\lambda\ge -n\right\}\right)\|
}\\
\shoveleft{
=\sup_{t\ge t_0}\|e^{tA}E_A\left(\left\{\lambda \in \sigma(A)\,
\middle|\,\Rep\lambda<-n\right\}\right)\|
}\\
\hfill
\text{by \eqref{boundedop};}
\\
\shoveleft{
=\sup_{t\ge t_0}
\left\|\int\limits_{\left\{\lambda\in\sigma(A) \,\middle|\, \Rep\lambda<-n\right\}}e^{t\lambda}\,dE_A(\lambda)\right\|
\le 4M \sup_{t\ge t_0}\sup_{\left\{\lambda\in\sigma(A) \,\middle|\, \Rep\lambda<-n\right\}}\left|e^{t\lambda}\right|
}\\
\ \ \
=4M \sup_{t\ge t_0}\sup_{\left\{\lambda\in\sigma(A) \,\middle|\, \Rep\lambda<-n\right\}}e^{t\Rep\lambda}
=4Me^{-t_0n}\to 0,\ n\to\infty,
\hfill
\end{multline*} 
where the constant $M\ge 1$ is from \eqref{bounded}.

Thus, for any $t_0>0$, the sequence of operator functions
\[
\left(E_A\left(\left\{\lambda\in\sigma(A) \,\middle|\, \Rep\lambda \ge  -n\right\}\right)e^{\cdot A_n}\right)_{n\in \N},
\]
continuous relative to the operator norm on $[0,\infty)$,
converges \textit{uniformly} to $T(\cdot)$ in the operator norm on $[t_0,\infty)$. This implies \textit{immediate norm continuity} for the semigroup $\left\{T(t)\right\}_{t\ge 0}$ and completes the proof.
\end{proof}

We directly obtain the following

\begin{cor}[Characterization of Immediately Norm-Continuous Scalar Type Spectral $C_0$-Semigroups]
A $C_0$-semigroup (of scalar type spectral operators) on a complex Banach space generated by a scalar type spectral operator $A$ is immediately norm-continuous iff, for any $b\in \R$, the set
\begin{equation*}
\left\{\lambda\in\sigma(A) \,\middle|\, \Rep\lambda\ge b\right\}
\end{equation*}
is bounded.
\end{cor}

By the \textit{Necessary Condition for the Generation of Eventually Norm-Continuous $C_0$-Semigroups} (Theorem \ref{NCENC}), we also instantly arrive at 

\begin{cor}[Equivalence of Eventual and Immediate Norm-Continuity]\label{EEINC}\ \\
If a scalar type spectral $C_0$-semigroup in a complex Banach space is eventually norm-continuous, it is immediately norm-continuous.
\end{cor}

\section{Differentiability}

Combining the characterizations of the \textit{immediate differentiability} and \textit{Gevrey ultradifferentiability}, in particular \textit{analyticity}, of scalar type spectral $C_0$-semigroups found in \cite{Markin2002(2),Markin2004(1),Markin2016} (see also \cite{Markin2008}) with the \textit{Generation Theorem} (Theorem \ref{GT}) for such semigroups,  we instantly obtain the subsequent characterizations of the generation of such semigroups.

\begin{thm}[Characterization of the Generation of Immediately Differentiable Scalar Type Spectral $C_0$-Semigroups]
A scalar type spectral operator $A$ in a complex Banach space generates an immediately differentiable $C_0$-semigroup (of scalar type spectral operators) iff there exists an $\omega\in \R$ and, for any $b>0$, there exists an $a\in \R$ such that
\begin{equation*}
\sigma(A)\subseteq \left\{\lambda\in \C \,\middle|\,\Rep\lambda\le \min \left(\omega,a-b\ln|\Imp\lambda|\right) \right\}.
\end{equation*}
\end{thm}

See \cite{Markin2004(1)} for details (cf. \cite{Pazy1968}).

\begin{thm}[Characterization of the Generation of Roumieu-type Gevrey Ultradifferentiable Scalar Type Spectral $C_0$-Semigroups]
Let $1\le \beta <\infty$. A scalar type spectral operator $A$ in a complex Banach space generates a $\beta$th-order Roumieu-type Gevrey ultradifferentiable on $(0,\infty)$ $C_0$-semigroup (of scalar type spectral operators) iff there exist $b>0$ and $a\in \R$ such that
\begin{equation*}
\sigma(A)\subseteq \left\{\lambda\in \C \,\middle|\,\Rep\lambda\le a-b|\Imp\lambda|^{1/\beta}\right\}.
\end{equation*}
\end{thm}

See \cite{Markin2004(1),Markin2008} (cf. \cite{CrPaz1969}).

In particular, for $\beta=1$, we arrive at

\begin{cor}[Characterization of the Generation of Analytic Scalar Type Spectral $C_0$-Semigroups]
A scalar type spectral operator $A$ in a complex Banach space generates an analytic $C_0$-semigroup (of scalar type spectral operators) iff there exist $b>0$ and $a\in \R$ such that
\begin{equation*}
\sigma(A)\subseteq \left\{\lambda\in \C \,\middle|\,\Rep\lambda\le a-b|\Imp\lambda|\right\}.
\end{equation*}
\end{cor}

See \cite{Markin2002(2)} (cf. \cite{Yosida1958}).

A similar characterization of the generation of the $\beta$th-order Beurling-type Gevrey ultradifferentiable on $(0,\infty)$ scalar type spectral $C_0$-semigroups  ($1<\beta<\infty$) is found in \cite{Markin2016}. 

For characterizations of the generation of the \textit{Carleman ultradifferentiable} scalar type spectral $C_0$-semigroups, see \cite{Markin2008,Markin2016}.

We obtain the following characterization of  the generation of \textit{eventually differentiable} scalar type spectral $C_0$-semigroups dropping the estimate on the resolvent of the semigroup's generator in its general Hille–Yosida type counterpart \cite[Theorem $2.5$]{Pazy1968}.

\begin{samepage}
\begin{thm}[Characterization of the Generation of Eventual Differentiable Scalar Type Spectral $C_0$-Semigroups]\label{CED} 
A scalar type spectral operator $A$ in a complex Banach space generates an eventually differentiable $C_0$-semigroup $\left\{T(t)\right\}_{t\ge 0}$ (of scalar type spectral operators) iff there exist $\omega\in \R$, $b>0$, and $a\in \R$ such that
\begin{equation}\label{ced1}
\sigma(A)\subseteq \left\{\lambda\in \C \,\middle|\,\Rep\lambda\le \min \left(\omega,a-b\ln|\Imp\lambda|\right) \right\},
\end{equation}
in which case the orbit maps
\begin{equation*}
[0,\infty)\ni t\mapsto T(t)f\in X
\end{equation*} 
are differentiable on $[b^{-1},\infty)$ for all $f\in X$.
\end{thm}
\end{samepage}

\begin{proof}\

{\it ``Only if''} part immediately follows from \cite[Theorem $2.5$]{Pazy1968}.

\smallskip
{\it ``If''} part. Suppose that, for a scalar type spectral operator $A$ in a complex Banach space $(X,\|\cdot\|)$  with spectral measure $E_A(\cdot)$, inclusion \eqref{ced1} holds with some $\omega\in \R$, $b>0$, and $a\in \R$.

To prove that the $C_0$-semigroup of exponentials 
$\left\{e^{tA}\right\}_{t\ge 0}$ generated by $A$ (see Preliminaries) is eventually differentiable is to show that
\begin{equation}\label{incl}
\exists\, t_0>0:\ e^{t_0A}X\subseteq D(A)
\end{equation}
\cite{Pazy1968}.

For arbitrary $f\in X$ and $g^*\in X^*$,
\begin{align*}
\int\limits_{\sigma(A)} |\lambda|e^{b^{-1}\Rep\lambda}\,dv(f,g^*,\lambda)
&=\int\limits_{\left\{\lambda\in\sigma(A) \,\middle|\, \Rep\lambda<0\right\}} |\lambda|e^{b^{-1}\Rep\lambda}\,dv(f,g^*,\lambda)\\
&+\int\limits_{\left\{\lambda\in\sigma(A) \,\middle|\, \Rep\lambda\ge 0\right\}} |\lambda|e^{b^{-1}\Rep\lambda}\,dv(f,g^*,\lambda).
\end{align*}

Indeed, 
\[
\int\limits_{\left\{\lambda\in\sigma(A) \,\middle|\, \Rep\lambda\ge 0\right\}} |\lambda|e^{b^{-1}\Rep\lambda}\,dv(f,g^*,\lambda)
\]
due to the boundedness of the set
\[
\left\{\lambda\in\sigma(A) \,\middle|\, \Rep\lambda\ge 0\right\},
\]
which follows from inclusion \eqref{ced1}, the continuity of the integrated function on $\C$, and the finiteness of the measure $v(f,g^*,\cdot)$.

Further, for any $f\in X$ and $g^*\in X^*$,
\begin{multline}\label{i}
\int\limits_{\left\{\lambda\in\sigma(A) \,\middle|\, \Rep\lambda<0\right\}} |\lambda|e^{b^{-1}\Rep\lambda}\,dv(f,g^*,\lambda)
\le 
\int\limits_{\sigma(A)}
\left(|\Rep\lambda|+|\Imp\lambda|\right)e^{b^{-1}\Rep\lambda}\,dv(f,g^*,\lambda)
\\
\hfill
\text{since, by \eqref{ced1}, for $\lambda\in \left\{\lambda\in\sigma(A) \,\middle|\, \Rep\lambda< 0\right\}$,}
\\
\hfill
\text{$\Rep\lambda <  0$ and $|\Imp\lambda|\le e^{b^{-1}(a-\Rep\lambda)}$;}
\\
\shoveleft{
\le \int\limits_{\left\{\lambda\in\sigma(A) \,\middle|\, \Rep\lambda<0\right\}}
\left[-\Rep\lambda+e^{b^{-1}(a-\Rep\lambda)}\right]e^{b^{-1}\Rep\lambda}\,
dv(f,g^*,\lambda)
}\\
\hfill
\text{since $x \le e^x$, $x\ge 0$;}
\\
\shoveleft{
\le \int\limits_{\left\{\lambda\in\sigma(A) \,\middle|\, \Rep\lambda<0\right\}}
\left[be^{b^{-1}(-\Rep\lambda)}
+e^{ab^{-1}}e^{b^{-1}(-\Rep\lambda)}\right]e^{b^{-1}\Rep\lambda}\,dv(f,g^*,\lambda)
}\\
\shoveleft{
=\left[b+e^{ab^{-1}}\right]\int\limits_{\left\{\lambda\in\sigma(A) \,\middle|\, \Rep\lambda<0\right\}} 1\,dv(f,g^*,\lambda)
\le \left[b+e^{ab^{-1}}\right]\int\limits_{\sigma(A)} 1\,dv(f,g^*,\lambda)
}\\
\shoveleft{
=\left[b+e^{ab^{-1}}\right]v(f,g^*,\sigma(A))
}\\
\hfill
\text{by \eqref{tv};}
\\
\ \ \,
\le 4M\left[b+e^{ab^{-1}}\right]\|f\|\|g^*\|.
\hfill
\end{multline}

Also, for any $f\in X$,
\begin{multline}\label{ii}
\sup_{\{g^*\in X^*\,|\,\|g^*\|=1\}}
\int\limits_{\left\{\lambda\in\sigma(A)\,\middle|\,|\lambda|e^{b^{-1}\Rep\lambda}>n\right\}}
|\lambda|e^{b^{-1}\Rep\lambda}\,dv(f,g^*,\lambda)
\\
\shoveleft{
\le \sup_{\{g^*\in X^*\,|\,\|g^*\|=1\}}\int\limits_{\left\{\lambda\in\sigma(A)\,\middle|\,\Rep\lambda<0,\ |\lambda|e^{b^{-1}\Rep\lambda}>n\right\}}
|\lambda|e^{b^{-1}\Rep\lambda}\,dv(f,g^*,\lambda)
}\\
\shoveleft{
+ \sup_{\{g^*\in X^*\,|\,\|g^*\|=1\}}\int\limits_{\left\{\lambda\in\sigma(A)\,\middle|\,\Rep\lambda\ge 0,\ |\lambda|e^{b^{-1}\Rep\lambda}>n\right\}}
|\lambda|e^{b^{-1}\Rep\lambda}\,dv(f,g^*,\lambda)
}\\
\hspace{1.2cm}
\to 0,\ n\to\infty.
\hfill
\end{multline}

Indeed, since, due to the boundedness of the set
\[
\left\{\lambda\in\sigma(A) \,\middle|\, \Rep\lambda\ge 0\right\}
\]
and the continuity of the integrated function 
on $\C$, the set
\[
\left\{\lambda\in\sigma(A)\,\middle|\,\Rep\lambda\ge 0,\ 
|\lambda|e^{b^{-1}\Rep\lambda}>n\right\}
\]
is \textit{empty} for all sufficiently large $n\in \N$,
we immediately infer that
\[
\lim_{n\to\infty}\sup_{\{g^*\in X^*\,|\,\|g^*\|=1\}}\int\limits_{\left\{\lambda\in\sigma(A)\,\middle|\,\Rep\lambda\ge 0,\ 
|\lambda|e^{b^{-1}\Rep\lambda}>n\right\}}
|\lambda|e^{b^{-1}\Rep\lambda}\,dv(f,g^*,\lambda)=0.
\]

Finally, as in \eqref{i}, for an arbitrary $f\in X$, we have:
\begin{multline*}
\sup_{\{g^*\in X^*\,|\,\|g^*\|=1\}}\int\limits_{\left\{\lambda\in\sigma(A)\,\middle|\,\Rep\lambda<0,\ |\lambda|e^{b^{-1}\Rep\lambda}>n\right\}}
|\lambda|e^{b^{-1}\Rep\lambda}\,dv(f,g^*,\lambda)
\\
\shoveleft{
\le\sup_{\{g^*\in X^*\,|\,\|g^*\|=1\}}
\left[b+e^{ab^{-1}}\right]\int\limits_{\left\{\lambda\in\sigma(A)\,\middle|\,\Rep\lambda<0,\ |\lambda|e^{b^{-1}\Rep\lambda}>n\right\}}
1\,dv(f,g^*,\lambda)
}\\
\hfill
\text{by \eqref{cond(ii)};}
\\
\shoveleft{
\le \left[b+e^{ab^{-1}}\right]\cdot
}\\
\shoveleft{
\cdot\sup_{\left\{g^*\in X^*\,|\,\|g^*\|=1\right\}}
4M\|E_A(\{\lambda\in\sigma(A)\,|\,\Rep\lambda<0,\ |\lambda|e^{b^{-1}\Rep\lambda}>n\})f\|\|g^*\|
}\\
\shoveleft{
=4M\left[b+e^{ab^{-1}}\right]\|E_A(\{\lambda\in\sigma(A)\,|\,\Rep\lambda<0,\ |\lambda|e^{b^{-1}\Rep\lambda}>n\})f\|
}\\
\hfill
\text{by the strong continuity of the {\it spectral measure};}
\\
\ \ \,
\to 4M\left[b+e^{ab^{-1}}\right]\left\|E_A\left(\emptyset\right)f\right\|=0,\ n\to \infty.
\hfill
\end{multline*}

By Proposition \ref{prop} and the properties of the \textit{operational calculus} (see {\cite[Theorem XVIII.$2.11$ (f)]{Dun-SchIII}}), \eqref{i} and \eqref{ii} jointly imply that, 
\[
f\in D\left(Ae^{b^{-1}A}\right),
\]
and hence, 
\begin{equation*}
e^{b^{-1}A}f\in D(A),
\end{equation*}
i.e., inclusion \eqref{incl} holds for $t_0=b^{-1}>0$, which implies that the orbit maps
\begin{equation*}
[0,\infty)\ni t\mapsto T(t)f\in X
\end{equation*} 
are differentiable on $[b^{-1},\infty)$ for all $f\in X$.
\end{proof}

Instantly, we obtain the following

\begin{cor}[Characterization of Eventually Differentiable Scalar Type Spectral $C_0$-Semigroups] A $C_0$-semigroup $\left\{T(t)\right\}_{t\ge 0}$ (of scalar type spectral operators) on a complex Banach space generated by a scalar type spectral operator $A$ is eventually differentiable iff there exist $b>0$ and $a\in \R$ such that
\begin{equation*}
\sigma(A)\subseteq \left\{\lambda\in \C \,\middle|\,\Rep\lambda\le a-b\ln|\Imp\lambda| \right\},
\end{equation*}
in which case the orbit maps
\begin{equation*}
[0,\infty)\ni t\mapsto T(t)f\in X
\end{equation*} 
are differentiable on $[b^{-1},\infty)$ for all $f\in X$.
\end{cor}

Since an eventually differentiable $C_0$-semigroup is eventually norm-continuous, by the \textit{Equivalence of Eventual and Immediate Norm-Continuity} (Corollary \ref{EEINC}), we also arrive at

\begin{cor}[Eventual Differentiability Implies Immediate Norm-Continuity]
If a scalar type spectral $C_0$-semigroup is eventually differentiable, it is immediately norm-continuous.
\end{cor}

\section{Compactness}

Let us first prove the following characterization of compactness for the scalar type spectral operators emerging directly from the classical \textit{Riesz-Schauder Theorem}
(see, e.g., \cite{Markin2020EOT}).

\begin{lem}[Characterization of Compactness for Scalar Type Spectral Operators]\label{CSTSO}
A scalar type spectral operator $A\neq 0$ in a complex infinite-dimensional Banach space is compact iff the spectrum $\sigma(A)$ of $A$ consists of $0$ and a nonempty countable set of nonzero eigenvalues with finite geometric multiplicities.

If the set $\sigma(A)\setminus \left\{0\right\}$ is countably infinite, for its arbitrary arrangement $\left\{ \lambda_n\right\}_{n=1}^\infty$ (i.e., $\N\ni n\mapsto \lambda_n\in \sigma(A)\setminus \left\{0\right\}$ is a bijection),
\begin{equation*}
\lambda_n\to 0,\ n\to \infty.
\end{equation*}
\end{lem}

\begin{proof}\

\textit{``Only if''} part follows directly from the \textit{Riesz-Schauder Theorem}
(see, e.g., \cite{Markin2020EOT}) in view of the fact that, for a scalar type spectral operator $A\neq 0$ with spectral measure $E_A(\cdot)$, due to the integral representation
\begin{equation*}
A=\int\limits_{\sigma(A)} \lambda\,dE_A(\lambda)
\end{equation*}
(see Preliminaries),
\[
\sigma(A)\setminus \left\{0\right\}\neq \emptyset.
\]

\textit{``If''} part. Suppose that, a scalar type spectral operator $A\neq 0$ in a complex Banach space $(X,\|\cdot\|)$ with spectral measure $E_A(\cdot)$ is subject to the conditions of the \textit{``if''} part. Then the spectrum $\sigma(A)$ is a \textit{bounded} set, which implies that $A$ is \textit{bounded} (see Preliminaries).

If the set 
\[
\sigma(A)\setminus \left\{0\right\}\subseteq \sigma_p(A)
\]
is finite, then
\[
\sigma(A)\setminus \left\{0\right\}=\left\{\lambda_1,\dots,\lambda_n  \right\}
\]
with some $n\in \N$. This implies, that
\begin{equation*}
A=\int\limits_{\sigma(A)} \lambda\,dE_A(\lambda)
=\sum_{k=1}^n \lambda_kE_A\left(\left\{ \lambda_k\right\}\right).
\end{equation*}

Whence, in view of the fact that each $\lambda_k$, $k=1,\dots,n$, is an eigenvalue with a finite geometric multiplicity, i.e., the subspace $E_A\left(\left\{ \lambda_k\right\}\right)X$,
$k=1,\dots,n$, is \textit{finite-dimensional},  we infer that the operator $A$ is of \textit{finite rank}, and hence, compact.

Now suppose that the set 
\[
\sigma(A)\setminus \left\{0\right\}\subseteq \sigma_p(A)
\]
is countably infinite and let $\left\{ \lambda_n\right\}_{n=1}^\infty$ is its arbitrary arrangement. Then, by the premise,
\[
\lambda_n\to 0,\ n\to \infty.
\]

Whence, by the properties of the Borel operational calculus \cite{Dun-SchIII}, 
we infer that, for any $n\in \N$,
\begin{multline*}
\left\|A-\sum_{k=1}^n \lambda_kE_A\left(\left\{ \lambda_k\right\}\right)\right\|
=\left\|\int\limits_{\sigma(A)\setminus \left\{ \lambda_1,\dots,\lambda_n\right\}} \lambda\,dE_A(\lambda)\right\|
\\
\hfill
\text{by \eqref{boundedop};}
\\
\ \
\le 4M\sup_{k\ge n+1}|\lambda_k|\to 0,\ n\to \infty.
\hfill
\end{multline*}

Therefore, the operator $A$ is \textit{compact} as the \textit{uniform limit} of the sequence 
\[
\left(\sum_{k=1}^n \lambda_kE_A\left(\left\{ \lambda_k\right\}\right)\right)_{n\in \N}
\]
of finite-rank operators (see, e.g., \cite{Markin2020EOT}).

With all possible cases considered, the proof of the \textit{``if''} part, and hence, of the entire statement, is complete.
\end{proof}

\begin{thm}[Characterization of the Generation of Immediately Compact Scalar Type Spectral $C_0$-Semigroups]\label{CIC}
A scalar type spectral operator $A\neq 0$ in a complex infinite-dimensional Banach space generates an immediately compact $C_0$-semigroup iff 
$\sigma(A)$ is a countably infinite set of eigenvalues with finite geometric multiplicities,
which has no finite limit points and is such that, for its arbitrary arrangement $\left\{ \lambda_n\right\}_{n=1}^\infty$,
\[
\Rep\lambda_n\to -\infty,\ n\to \infty.
\]
\end{thm}

\begin{proof}\

\textit{``Only if''} part. Suppose that, a scalar type spectral operator $A\neq 0$ in a complex Banach space $(X,\|\cdot\|)$ with spectral measure $E_A(\cdot)$ generates 
an \textit{immediately compact} $C_0$-semigroup $\left\{T(t)\right\}_{t\ge 0}$. Then $A$ is necessarily \textit{unbounded} (see Remarks \ref{remscs}), and hence, its spectrum 
$\sigma(A)$ is an unbounded set (see Preliminaries), which, in particular, is \textit{infinite}.

Then, by \cite[Theorem $3.3$]{Pazy1968}, the semigroup is \textit{immediately norm-continuous} and the resolvent
\[
R(\lambda,A):=(A-\lambda I)^{-1}=\int\limits_{\sigma(A)}\dfrac{1}{\mu-\lambda}\,dE_A(\mu)
\]
is \textit{compact} for all $\lambda \in \rho(A):=\C\setminus \sigma(A)$, i.e., the operator $A$ is \textit{regular} in the sense of \cite[Definition 1]{Schwartz1954}.

By \cite[Lemma 1]{Schwartz1954} and \cite[Corollary V.$1.15$]{Engel-Nagel2006}, 
\[
\sigma(A)=\sigma_p(A)
\]
is a countable set with no finite limit points. This, in view of the fact that the spectrum $\sigma(A)$ is an \textit{infinite} set, implies that, for its arbitrary arrangement $\left\{ \lambda_n\right\}_{n=1}^\infty$,
\[
\lambda_n\to \infty,\ n\to \infty.
\]

Since $T(t)=e^{tA}$, $t>0$, (see Preliminaries), in particular $T(1)=e^A$, is a \textit{compact} operator and by the \textit{Point Spectrum Spectral Mapping Theorem} (\cite[Theorem $4.2$]{Markin2021(2)})
\begin{equation*}
\sigma_p(T(1)) =e^{\sigma_p(A)},\  t\ge 0,
\end{equation*}
(cf. \cite[Theorem V.$2.6$]{Engel-Nagel2006}), by the \textit{Characterization of Compactness for Scalar Type Spectral Operators} (Lemma \ref{CSTSO}), we infer that
\[
e^{\Rep\lambda_n}=\left|e^{\lambda_n}\right|\to 0,\ n\to\infty,
\]
and hence,
\[
\Rep\lambda_n\to -\infty,\ n\to \infty.
\]

Furthermore, for an arbitrary $\mu \in \rho(A)$,
\begin{equation}\label{rps1}
\lambda\in \sigma_p(A) \iff (\lambda-\mu)^{-1}\in \sigma_p\left(R(\mu,A)\right),
\end{equation}
in which case
\begin{equation}\label{rps2}
\ker(A-\lambda I)=\ker\left(R(\mu,A)-(\lambda-\mu)^{-1}I\right)
\end{equation}
(see, e.g., \cite{Markin2020EOT}).

This, by the \textit{Characterization of Compactness for Scalar Type Spectral Operators} (Lemma \ref{CSTSO}), implies that each eigenvalue of $A$ has a finite geometric multiplicity.

\textit{``If''} part. Suppose that, a scalar type spectral operator $A\neq 0$ in a complex Banach space $(X,\|\cdot\|)$ with spectral measure $E_A(\cdot)$ is subject to the conditions of the \textit{``if''} part. 

Since the spectrum $\sigma(A)$ is unbounded, the operator $A$ is \textit{unbounded} 
(see Preliminaries).

By the \textit{Characterization of the Generation of Immediately Norm-Continuous Scalar Type Spectral $C_0$-Semigroups} (Theorem \ref{CGINCS}), we infer that $A$ generates an \textit{immediately norm-continuous} $C_0$-semigroup. 

Furthermore, by the \textit{Spectral Mapping Theorem for the Resolvent} (\cite[Theorem V.$1.13$]{Engel-Nagel2006}), for an arbitrary $\lambda\in \rho(A)$,
\[
\sigma\left(R(\lambda,A)\right)\setminus \left\{ 0\right\}=\left(\sigma(A)-\lambda\right)^{-1},
\]
and hence, the spectrum $\sigma\left(R(\lambda,A)\right)$ of $R(\lambda,A)$ consists of $0$ (because $A$ is unbounded) and a nonempty countably infinite set $\sigma\left(R(\lambda,A)\right)\setminus \left\{0\right\}$ such that, for its arbitrary arrangement $\left\{ \lambda_n\right\}_{n=1}^\infty$,
\begin{equation*}
\lambda_n\to 0,\ n\to \infty.
\end{equation*}

By \eqref{rps1} and \eqref{rps2}, we also infer that the set set $\sigma\left(R(\lambda,A)\right)\setminus \left\{0\right\}$ consists of the eigenvalues of
$R(\lambda,A)$ with finite geometric multiplicities.

By the the \textit{Characterization of Compactness for Scalar Type Spectral Operators} (Lemma \ref{CSTSO}), we conclude that the resolvent $R(\lambda,A)$ is \textit{compact} for all $\lambda\in \rho(A)$, which along with the \textit{immediate norm-continuity} of the $C_0$-semigroup generated by $A$, by \cite[Theorem $3.3$]{Pazy1968}, implies that the latter is \textit{immediately compact}.
\end{proof}

Instantly, we obtain

\begin{cor}[Characterization of Immediately Compact Scalar Type Spectral $C_0$-Semigroups]
A $C_0$-semigroup (of scalar type spectral operators) on complex infinite-dimensional Banach space generated by a scalar type spectral operator $A$ is immediately compact iff $\sigma(A)$ is a countably infinite set of eigenvalues with finite geometric multiplicities, which has no finite limit points and is such that, for its arbitrary arrangement $\left\{ \lambda_n\right\}_{n=1}^\infty$,
\[
\Rep\lambda_n\to -\infty,\ n\to \infty.
\]
\end{cor}

We also have the following analogue of the  \textit{Equivalence of Eventual and Immediate Norm-Continuity} (Corollary \ref{EEINC}).

\begin{cor}[Equivalence of Eventual and Immediate Compactness]\label{EIC}\ \\
If a scalar type spectral $C_0$-semigroup in a complex infinite-dimensional Banach space is eventually compact, its is immediately compact.
\end{cor}

\begin{proof}
Suppose that a $C_0$-semigroup $\left\{T(t)\right\}_{t\ge 0}$ on a complex infinite-dimensional Banach space generated by a scalar type spectral operator $A\neq 0$ with spectral measure $E_A(\cdot)$ is \textit{compact} for $t\ge t_0$ with some $t_0>0$.

Since $T(t)=e^{tA}$, $t\ge t_0$,  (see Preliminaries) is a \textit{compact operator}, in view of the \textit{infinite dimensionality} of the underlying space, its inverse
\[
T^{-1}(t)=e^{-tA}:=\int\limits_{\sigma(A)} e^{-t\lambda}\,dE_A(\lambda)
\]
is necessarily \textit{unbounded} (see, e.g., \cite{Markin2020EOT}), which implies that
the scalar type spectral operator $A$ has an \textit{unbounded}, in particular \textit{infinite}, spectrum $\sigma(A)$, and hence, is \textit{unbounded} 
(see Preliminaries).

Since, by the \textit{Weak Spectral Mapping Theorem} (\cite[Theorem $4.1$]{Markin2021(2)}),
\begin{equation}\label{wsmt1}
\sigma(T(t))=\overline{e^{t\sigma(A)}},\ t\ge 0,
\end{equation}
we infer that $\sigma(A)$ has no finite limit points. Otherwise, by the \textit{continuity} of the exponential function $e^{t_0\cdot}$, the spectrum $\sigma(T(t_0))$ of the \textit{compact operator} $T(t_0)$ would have a \textit{nonzero finite limit point}, which is impossible by the \textit{Characterization of Compactness for Scalar Type Spectral Operators} (Lemma \ref{CSTSO}).

Being an infinite set in $\C$ with no finite limit points, the spectrum $\sigma(A)$ is \textit{countably infinite}.

Further, since $\sigma(A)$ is the \textit{support} for the spectral measure $E_A(\cdot)$ (see Preliminaries) and consist of \textit{isolated points},
\[
\forall\, \lambda\in \sigma(A):\ E_A(\lambda)\neq 0,
\]
and hence,
\[
\sigma(A)=\sigma_p(A),
\]
\cite{Dun-SchIII}.

Since the operator $T(t)=e^{tA}$ is \textit{compact} for each $t\ge t_0$, considering
the $2\pi i$-\textit{periodicity} ($i$ is the \textit{imaginary unit}) of the exponential function $e^{\cdot}$, we infer that there exists a $t_1\ge t_0$ such that the spectrum $\sigma(T(t_1))$ is an \textit{infinite} set.
Indeed, assuming that the set 
\[
\sigma(T(t_0))=\overline{e^{t_0\sigma(A)}}
\]
(see \eqref{wsmt1}) is \textit{finite}, we conclude that the \textit{infinite set} $\sigma(A)$ contains an infinitely subset consisting of the points of the form 
\[
\lambda_0+it_0^{-1}2\pi n,\ n\in \Z,
\]
with some $\lambda_0\in \C$, mapped under $e^{t_0\cdot}$ into the same point of 
$\sigma(T(t_0))$. Then, for any $t>t_0$, these points are mapped under $e^{t\cdot}$
into infinitely many pints of
\[
\sigma(T(t))=\overline{e^{t\sigma(A)}}
\]
(see \eqref{wsmt1}).

Let $\left\{ \lambda_n\right\}_{n=1}^\infty$ be an arbitrary arrangement of $\sigma(A)$. Then, by the \textit{Characterization of Compactness for Scalar Type Spectral Operators} (Lemma \ref{CSTSO}), 
\[
e^{t_1\Rep\lambda_n}=\left|e^{t_1\lambda_n}\right|\to 0,\ n\to\infty.
\]

Therefore,
\[
\Rep\lambda_n\to -\infty,\ n\to \infty.
\]

Furthermore, since, for arbitrary $\lambda\in \sigma_p(A)$ and $f\in \ker(A-\lambda I)$,
\[
y(t):=e^{t\lambda}f,\ t\ge 0,
\]
is the \textit{unique} eigenvalue solution of the Cauchy problem
\begin{equation*}
\begin{cases}
y'(t)=Ay(t),\ t\ge 0,\\
y(0)=f,
\end{cases}
\end{equation*}
(see Preliminaries), we infer that
\[
T(t)f=e^{t\lambda}f,\ t\ge 0,
\]
and hence,
\[
f\in \ker\left(T(t)-e^{t\lambda} I\right),\ t\ge 0,
\]
in particular, $f\in \ker\left(T(t_0)-e^{ t_0\lambda} I\right)$, which implies that
\[
\ker(A-\lambda I)\subseteq \ker\left(T(t_0)-e^{t_0\lambda} I\right).
\]

Whence, since, by the \textit{Characterization of Compactness for Scalar Type Spectral Operators} (Lemma \ref{CSTSO}), $\ker\left(T(t_0)-e^{t_0\lambda} I\right)$ is \textit{finite-dimensional}, we conclude that each eigenvalue of $A$ has a \textit{finite geometric multiplicity}.

Thus, by the \textit{Characterization of the Generation of Immediately Compact Scalar Type Spectral $C_0$-Semigroups} (Theorem \ref{CIC}), the $C_0$-semigroup generated by $A$ is \textit{immediately compact}. 
\end{proof}

\section{The Case of Normal $C_0$-Semigroups}

For the important particular case of the $C_0$-semigroups of \textit{normal operators}, we instantly obtain the following counterparts of the corresponding characterizations
established for $C_0$-semigroups of \textit{scalar type spectral operators}.

\begin{cor}[Characterization of the Generation of Immediately Norm-Continuous Normal $C_0$-Semigroups]
A normal operator $A$ in a complex Hilbert space generates an immediately norm continuous $C_0$-semigroup (of normal operators) iff, for any $b\in \R$, the set
\begin{equation*}
\left\{\lambda\in\sigma(A) \,\middle|\, \Rep\lambda\ge b\right\}
\end{equation*}
is bounded.
\end{cor}
 
\begin{rem}
In view of the fact that, for a normal operator $A$ in a complex Hilbert space,
\begin{equation*}
\|R(\lambda,A)\|=\dfrac{1}{\dist(\lambda,\sigma(A))},\ \lambda\in \rho(A),
\end{equation*}
where
\[
\dist(\lambda,\sigma(A)):=\inf_{\mu\in \sigma(A)}|\mu-\lambda|>0,
\]
(see, e.g., \cite{Dun-SchII,Plesner}),
the prior corollary is consistent with the known characterizations of immediate and eventual norm continuity of $C_0$-semigroups in complex Hilbert spaces (\cite[Theorem 1]{You1992} and \cite[Theorem 4]{El-Mennaoui-Engel1994}, respectively).
\end{rem} 

\begin{cor}[Characterization of Immediately Norm-Continuous Normal $C_0$-Se\-migroups]
A $C_0$-semigroup (of normal operators) on a complex Hilbert space generated by a normal operator $A$ is immediately norm-continuous iff, for any $b\in \R$, the set
\begin{equation*}
\left\{\lambda\in\sigma(A) \,\middle|\, \Rep\lambda\ge b\right\}
\end{equation*}
is bounded.
\end{cor}

\begin{cor}[Equivalence of Eventual and Immediate Norm-Continuity]\ \\
If a normal $C_0$-semigroup in a complex Hilbert space is eventually norm-continuous, it is immediately norm-continuous.
\end{cor}

\begin{cor}[Characterization of the Generation of Roumieu-type Gevrey Ultradifferentiable Normal $C_0$-Semigroups]
Let $1\le \beta <\infty$. A normal operator $A$ in a complex Hilbert space generates a $\beta$th-order Roumieu-type Gevrey ultradifferentiable on $(0,\infty)$ $C_0$-semigroup (of normal operators) iff there exist $b>0$ and $a\in \R$ such that
\begin{equation*}
\sigma(A)\subseteq \left\{\lambda\in \C \,\middle|\,\Rep\lambda\le a-b|\Imp\lambda|^{1/\beta}\right\}.
\end{equation*}
\end{cor}

\begin{cor}[Characterization of the Generation of Beurling-type Gevrey Ultradifferentiable Normal $C_0$-Semigroups]
Let $1<\beta<\infty$. A normal operator $A$ in a complex Hilbert space generates a $\beta$th-order Beurling-type Gevrey ultradifferentiable on $(0,\infty)$
$C_0$-semigroup (of normal operators) iff, for any $b>0$, there exists an $a\in \R$ such that
\begin{equation*}
\sigma(A)\subseteq \left\{\lambda\in \C \,\middle|\,\Rep\lambda\le a-b|\Imp\lambda|^{1/\beta}\right\}.
\end{equation*}
\end{cor}

Cf. \cite[Corollary $4.1$]{Markin2016}.

\begin{cor}[Characterization of the Generation of Eventually Differentiable Normal $C_0$-Semigroups]
A normal operator $A$ in a complex Hilbert space generates an eventually differentiable $C_0$-semigroup $\left\{T(t)\right\}_{t\ge 0}$ (of normal operators) iff there exist $\omega\in \R$, $a>0$, and $b>0$ such that  
\begin{equation*}
\sigma(A)\subseteq \left\{\lambda\in \C \,\middle|\,\Rep\lambda\le \Rep\lambda\le \min\left(\omega,a-b\ln|\Imp\lambda|\right) \right\},
\end{equation*}
in which case the orbits
\begin{equation*}
[0,\infty)\ni t\mapsto T(t)f\in X
\end{equation*} 
are strongly differentiable on $[b^{-1},\infty)$ for all $f\in X$.
\end{cor}

\begin{cor}[Characterization of Eventually Differentiable Normal $C_0$-Semigroups]
A $C_0$-semigroup $\left\{T(t)\right\}_{t\ge 0}$ (of normal operators) on a complex Hilbert space generated by a normal operator $A$ is eventually differentiable 
iff there exist $a>0$ and $b>0$ such that  
\begin{equation*}
\sigma(A)\subseteq \left\{\lambda\in \C \,\middle|\,\Rep\lambda\le a-b\ln|\Imp\lambda| \right\},
\end{equation*}
in which case the orbits
\begin{equation*}
[0,\infty)\ni t\mapsto T(t)f\in X
\end{equation*} 
are strongly differentiable on $[b^{-1},\infty)$ for all $f\in X$.
\end{cor}

\begin{cor}[Characterization of Compactness for Normal Operators]\ \\
A normal operator $A\neq 0$ in a complex infinite-dimensional Hilbert space is compact iff the spectrum $\sigma(A)$ of $A$ consists of $0$ and a nonempty countable set of nonzero eigenvalues with finite geometric multiplicities.

If the set $\sigma(A)\setminus \left\{0\right\}$ is countably infinite, for its arbitrary arrangement $\left\{ \lambda_n\right\}_{n=1}^\infty$ (i.e., $\N\ni n\mapsto \lambda_n\in \sigma(A)\setminus \left\{0\right\}$ is a bijection),
\begin{equation*}
\lambda_n\to 0,\ n\to \infty.
\end{equation*}
\end{cor}

\begin{cor}[Characterization of the Generation of Immediately Compact Normal $C_0$-Semigroups]
A normal operator $A\neq 0$ in a complex infinite-dimensional Hilbert space generates an immediately compact $C_0$-semigroup (of normal operators) iff 
$\sigma(A)$ is a countably infinite set of eigenvalues with finite geometric multiplicities,
which has no finite limit points and is such that, for its arbitrary arrangement $\left\{ \lambda_n\right\}_{n=1}^\infty$,
\[
\Rep\lambda_n\to -\infty,\ n\to \infty.
\]
\end{cor}

\begin{cor}[Characterization of Immediately Compact Normal $C_0$-Semigroups]
A $C_0$-semigroup (of normal operators) on a complex infinite-dimensional Hilbert space generated by a normal operator $A$ is immediately compact iff $\sigma(A)$ is a countably infinite set of eigenvalues with finite geometric multiplicities, which has no finite limit points and is such that, for its arbitrary arrangement $\left\{ \lambda_n\right\}_{n=1}^\infty$,
\[
\Rep\lambda_n\to -\infty,\ n\to \infty.
\]
\end{cor}

\begin{cor}[Equivalence of Eventual and Immediate Compactness]\ \\
If a normal $C_0$-semigroup in a complex infinite-dimensional Hilbert space is eventually compact, its is immediately compact.
\end{cor}



\end{document}